\documentclass{article}
\usepackage{microtype}
\usepackage{fullpage}
\usepackage{graphicx}
\usepackage{hyperref}

\usepackage{amsmath,amsfonts,bm}

\def\1{\bm{1}}

\def\vn{{\bm{n}}}

\def\vu{{\bm{u}}}

\def\vx{{\bm{x}}}

\DeclareMathAlphabet{\mathsfit}{\encodingdefault}{\sfdefault}{m}{sl}
\SetMathAlphabet{\mathsfit}{bold}{\encodingdefault}{\sfdefault}{bx}{n}

\def\gA{{\mathcal{A}}}

\def\gD{{\mathcal{D}}}

\def\gS{{\mathcal{S}}}

\def\sN{{\mathbb{N}}}

\def\sR{{\mathbb{R}}}

\newcommand{\E}{\mathbb{E}}

\usepackage[utf8]{inputenc} 
\usepackage[T1]{fontenc}    
\usepackage{url}            
\usepackage{booktabs}       
\usepackage{amsfonts}       
\usepackage{nicefrac}       
\usepackage{microtype}      
\usepackage{mathrsfs}
\usepackage{multirow}
\usepackage{algorithm}
\usepackage{algorithmic}
\usepackage{adjustbox}
\usepackage{bbm}
\usepackage{bm}
\usepackage{color}
\usepackage{mathrsfs}
\usepackage{amsmath} 
\usepackage[numbers]{natbib}
\usepackage{float}
\usepackage{diagbox}
\usepackage{subcaption}
\usepackage{amssymb}
\usepackage[dvipsnames]{xcolor}
\usepackage{enumitem}
\newtheorem{thm}{Theorem}[section]
\newtheorem{defn}[thm]{Definition}

\newtheorem{remark}[thm]{Remark}

\newtheorem{prop}{Proposition}

\newenvironment{proof}{Proof:}{\hfill$\square$}
\usepackage[toc,page]{appendix}
\usepackage{tikz}
\usepackage{blkarray}
\newcommand{\norm}[1]{\left\|#1\right\|}

\begin{document}
\title{On the Convergence of Policy in \\ Unregularized Policy Mirror Descent}
\author{Dachao Lin\thanks{Academy for Advanced Interdisciplinary Studies, Peking University. \texttt{lindachao@pku.edu.cn}}
\and
Zhihua Zhang\thanks{School of Mathematical Sciences, Peking University. \texttt{zhzhang@math.pku.edu.cn}}
}

\maketitle

\begin{abstract}
	In this short note, we give the convergence analysis of the policy in recent famous policy mirror descent (PMD)\cite{lan2021policy,zhan2021policy,li2022homotopic,lan2022block,xiao2022convergence}. We mainly consider the unregularized setting following \cite{xiao2022convergence} with generalized Bregman divergence. The difference is that we directly give the convergence rates of policy under generalized Bregman divergence. Our results are inspired by the convergence of value function in previous works and are an extension study of policy mirror descent. Though some results have already appeared in previous work \citet{khodadadian2021linear,li2022homotopic}, we further discover a large body of Bregman divergences could give finite-step-convergence to an optimal policy, such as the classical Euclidean distance.
\end{abstract}

\section{Introduction}
Recently, many works have focused on the policy gradient descent. We mainly follow the tabular setting of policy gradient descent.
One line of work consider entropy-regularized MDP to give at least linear convergence, e.g., \citet{cen2021fast,zhan2021policy}. Particularly, \citet{zhan2021policy} extent common KL divergence to more general Bregman divergences. But the problem they studied is not the original MDP problem. 
Thus, this way needs a predefined precision of error to decide the small regularization terms, as well as a new round of algorithm if the precision changes.
Another line of works inherit the benefit of regularization. These literature solve the original problem directly with adaptive vanish regularization, e.g., \citet{lan2021policy,li2022homotopic}. 
This way still enjoys the linear rate if the step size is exponential, and directly solves the original problem. 
Moreover, \citet{li2022homotopic} also gave superlinear last-iterate convergence of policy under KL divergence, and our work has similar technique as theirs. 
Recently, \citet{xiao2022convergence} made progress on the unregularized setting. He showed that the intrinsic of linear rates are weighted Bregman divergence and exponential learning rate, instead of the regularization terms. However, \citet{xiao2022convergence} only claimed asymptotic convergence to an optimal policy based on the value convergence \cite{lan2021policy,zhan2021policy,agarwal2021theory}. 
Invoked from the work of \cite{xiao2022convergence,li2022homotopic}, we guess the last-iterate convergence of policy without regularization may still exists, which leads to our study of the explicit convergence rate of policy under generalized Bregman divergence in this note.

\section{Preliminaries}
We denote $[d] = \{ 1, \dots, d \}, d \geq 1$, and define the support set $\mathrm{supp}(p) = \{ i | p_i \neq 0, \forall i \in [d] \}$ for a vector $p \in \sR^d$, $\mathrm{supp}(p)^c = \sR^d\backslash \mathrm{supp}(p)$.
For a set $C$ and $x \in C$, we define the normal cone of $C$ at $x$ as $N_C(x):=\{g | g^\top (y - x) \leq 0, \forall y \in C \}$.
We use $ \mathrm{int}(\gS), \mathrm{ri}(\gS), \mathrm{cl}(\gS)$ to denote the interior, the relative interior and the closure of the set $\gS$, and $\partial \gS:=\mathrm{cl}(\Delta(\gA))\backslash \mathrm{ri}(\Delta(\gA))$ as the relative boundary of $\gS$.
We adopt $\nabla f(\cdot), \partial f(\cdot)$ as the gradient and subgradients of a fucntion $f(\cdot)$.

\subsection{Notation of MDP}
Most of the notation refers to \cite{xiao2022convergence}, and we strongly encourage readers to carefully read the profound work \cite{xiao2022convergence}.
A Markov decision process (MDP) can be specified with five elements as $(\gS,\gA,P,R,\gamma)$, where $\gS, \gA$ are finite state space and action space, with cardinalities $|\gS|$ and $|\gA|$, and $P$ is a transition probability function with $P (s'|s, a)$ denoting the probability of transitioning to $s'$ after taking action $a$ from state $s$. Reward function is $R \colon S \times A \to [0, 1]$ with entries of $R_{s,a}$ at state $s$ and action $a$, and $\gamma \in [0, 1)$
is a discount factor applied to the reward one-step in the future.
Let $\Delta(\gA)$ denote the probability simplex defined over the state space $\gA$,
\[ \Delta(\gA) := \left\{ \left. p \in \sR^{|\gA|}  \right| \sum_{a \in\gA} p_a = 1, p_a \geq 0, \forall a \in \gA \right\}. \]
The set of policy space is defined as 
\[ \Pi:= \Delta(\gA)^{|\gS|} = \left\{ \left. \pi= \{\pi_s\}_{s \in \gS} \right| \pi_s \in \Delta(\gA), \forall s \in \gS \right\}. \]
The transition matrix $P^{\pi} \colon \Pi \to \sR^{\gS\times \gS} $ under policy $\pi$ is
\[ P_{s,s'}(\pi) = \sum_{a \in \gA} \pi_{s, a} P(s'|s,a), \forall s,s' \in \gS. \]
Here $\pi_{s,a}$ is the probability for choosing action $a$ based on $\pi_s$.
The reward $r^{\pi} \colon \Pi \to \sR^{\gS} $ under the policy $\pi$ is
\[ r_s(\pi) = \sum_{a \in \gA} R_{s, a}\pi_{s,a}, \forall s \in \gS. \]
Value function of policy $\pi$ at each state is:
\[ V(\pi) = \sum_{t=0}^{\infty} \gamma^tP(\pi)^tr(\pi) = \left(I-\gamma P(\pi)\right)^{-1}r(\pi). \]
And the corresponding value function under an initial state distribution $\rho \in \Delta(\gS)$
is 
\[ V_{\rho}(\pi) = \sum_{s \in \gS} \rho_s V_s(\pi)= \rho^\top\left(I-\gamma P(\pi)\right)^{-1}r(\pi). \]
The state-action value function is
\begin{equation}\label{eq:Q}
	Q_{s,a}(\pi) = R_{s,a} + \gamma \sum_{s'\in \gS} P(s'|s,a)V(s').
\end{equation}
It is easy to see
\[ V_s(\pi) = \langle Q(s,\cdot), \pi(s,\cdot)\rangle, \forall s \in \gS.  \]

The optimal value function $V^* $ guarantees $V_s^* \leq V_s(\pi), \forall s  \in \gS$, e.g., \cite{puterman2014markov}.
We denote the corresponding optimal policy set as $\Pi^*:=\{\pi| V(\pi) =V^*\}$ (or $\Pi_s^*:=\{\pi_s| V(\pi) =V^*\}, \forall s \in \gS$), and the distance of a policy $\pi$ to $\Pi^*$ as
\[ \norm{\pi-\Pi^*}_{\infty} = \min_{\pi^* \in \Pi^*} \norm{\pi-\pi^*}_{\infty} = \min_{\pi^* \in \Pi^*} \max_{s \in \gS}\norm{\pi_s-\pi_s^*}_1. \]

We also need the discounted state-visitation distribution $d_s(\pi) \in \Delta(\gS)$ with entries defined below:
\[ d_{s,s'}(\pi) = (1-\gamma)\sum_{t=0}^\infty \gamma^t Pr^{\pi}(s_t=s'|s_0=s) = (1-\gamma)e_{s}^\top\left(I-\gamma P(\pi)\right)^{-1}e_{s'}, \forall s, s' \in \gS. \]
Similarly, for an initial state distribution $\rho \in \Delta(\gS)$, we define 
\[ d_{\rho,s'}(\pi) = (1-\gamma)\rho^\top \left(I-\gamma P(\pi)\right)^{-1}e_{s'}, \forall s' \in \gS. \]

We may need the distribution mismatch coefficient of the distribution $\rho$ from $\mu$ as
\[ \norm{\frac{\rho}{\mu}}_{\infty} :=\max_{s \in \gS} \frac{\rho_s}{\mu_s} \geq 1. \]

\subsection{Policy Mirror Descent}
Let $h \colon \gD \to \sR$ be a proper convex function on $\gD := \mathrm{dom}(h)$ where $\Delta(\gA) \subseteq \gD$, and continuously differentiable on $\mathrm{ri}(\Delta(\gA))$. 
Then we can define Bregman divergence as below
\[ D(p, p') := h(p)-h(p')-\langle \nabla h(p'), p-p'\rangle, \forall p,p' \in\Delta(\gA), \]
From the convexity of $h(\cdot)$, we can see $D(p,p')$ is nonnegative, and convex related to $p$.
Now we give Policy Mirror Descent (PMD) starting from $\pi^{(0)} \in \mathrm{ri}(\Pi)$ as below:
\begin{equation}\label{eq:pmd-pi} 
\pi_s^{(k+1)} = \mathop{\arg\min}_{p \in \Delta(\gA)} \left\{ \eta_k \langle Q_s(\pi^{(k)}), p\rangle +D(p, \pi_s^{(k)}) \right\}, \forall s \in \gS. 
\end{equation}

Following \cite[Lemma 6]{xiao2022convergence}, we have for Legendre type function $h(\cdot)$, $\nabla h(\pi_s^{(k)})$ is well-defined for all $k \geq0$ and $s \in \gS$.
Then from \cite[Theorem 27.4]{rockafellar2015convex}, the optimal condition of Eq.~\eqref{eq:pmd-pi} is 
\begin{equation}\label{eq:opt-cond1}
	\forall p \in \Delta(\gA), \langle \eta_k Q_s(\pi^{(k)}) + \nabla h(\pi_s^{(k+1)}) - \nabla h(\pi_s^{(k)}), p-\pi_s^{(k+1)} \rangle \geq 0,
\end{equation}
Hence, we obtain
\begin{equation}\label{eq:opt-cond}
\nabla h(\pi_s^{(k)}) - \nabla h(\pi_s^{(k+1)})  - \eta_k Q_s(\pi^{(k)}) \in N_{\Delta(\gA)}(\pi_s^{(k+1)}).
\end{equation}
where the normal cone $N_{\Delta(\gA)}(p), \forall p \in\Delta(\gA)$ is defined as below:
\begin{equation}\label{eq:normal-cone}
N_{\Delta(\gA)}(p) = \left\{ \vn = (n_1, \dots, n_{|\gA|}) \left. \right| n_i \leq n_j=n_k, \forall i \not\in \mathrm{supp}(p), \forall j, k \in \mathrm{supp}(p) \right\}.
\end{equation}

\section{Sublinear/Linear Convergence of Value Function}
The results in this section are all inherited from the impressive work of \cite{xiao2022convergence}.
First, we have the sublinear rate under constant step sizes as below.
\begin{thm}[\cite{xiao2022convergence} Theorem 8]
	Consider the generalized policy mirror descent method with $\pi(0) \in \mathrm{ri}(\Pi)$ and constant step size $\eta_k = \eta>0, \forall k \geq 0$. For any $\rho \in \Delta(\gS)$, we have for all $k \geq 0$,
	\[ V_{\rho}(\pi^{(k)})-V_{\rho}^* \leq \frac{1}{k+1}\left(\frac{D_0^*}{\eta(1-\gamma)}+\frac{1}{(1-\gamma)^2}\right). \]
\end{thm}
Furthermore, by defining $D_k^* :=D_{d_{\rho}(\pi^*)}(\pi^*,\pi^{(k)}) = \sum_{s\in\gS} d_{\rho,s}(\pi^*)D(\pi_s^*,\pi_s^{(k)})$, we have the linear rate under exponential step size as below.
\begin{thm}[\cite{xiao2022convergence} Theorem 10]
	Consider the generalized policy mirror descent method with $\pi(0) \in \mathrm{ri}(\Pi)$. Suppose
	the step sizes satisfy $\eta_0 > 0$ and
	\begin{equation}\label{eq:exp-lr}
	\eta_{k+1} \geq \frac{\vartheta_\rho}{\vartheta_\rho-1}\eta_k, \forall k\geq 0, \text{ with } \vartheta_\rho := \frac{1}{1-\gamma} \norm{\frac{d_{\rho}(\pi^*)}{\rho}}_{\infty}.
	\end{equation}
	Then we have for each $k \geq 0$,
	\[ V_{\rho}(\pi^{(k)})-V_{\rho}^* \leq \left(1-\frac{1}{\vartheta_\rho}\right)^k\left(\frac{1}{1-\gamma}+\frac{D_0^*}{\eta_0\gamma}\right). \]
\end{thm}

Moreover, the convergence of value function $V_{\rho}(\pi^{(k)})$ can be converted to the convergence of $Q_{s,a}(\pi^{(k)})$.
\begin{prop}\label{prop:Q-V}
	Denote the distribution mismatch ratio between the initial state distribution $\rho$ and the probability of  transitioning $P(\cdot|s,a)$ as $r_{\rho} := \max_{s',s \in \gS,a \in \gA}\frac{P(s'|s,a)}{\rho_{s'}} = \max_{s \in \gS,a \in \gA}\norm{\frac{P(\cdot|s,a)}{\rho}}_{\infty}  $. Then we have 
	\[ Q_{s,a}(\pi)-Q^*_{s, a} \leq \gamma \cdot r_{\rho}\left(V_{\rho}(\pi)-V_{\rho}^*\right). \]
\end{prop}
\begin{proof}
	From Eq.~\eqref{eq:Q}, we have that
	\begin{equation}
		Q_{s,a}(\pi^{(k)})-Q^*_{s, a} = \gamma \sum_{s'\in \gS} \frac{P(s'|s,a)}{\rho_{s'}} \cdot \rho_{s'}\left[V_{s'}(\pi^{(k)})-V^*_{s'}\right] \leq \gamma \cdot r_{\rho} \sum_{s' \in \gS} \rho_{s'}\left[V(s')-V^*(s')\right] = \gamma \cdot r_{\rho} \left(V_{\rho}(\pi)-V_{\rho}^*\right).
	\end{equation}
\end{proof}

\begin{remark}
	Such derivation in Proposition \ref{prop:Q-V} also appears in the proof of \cite[Theorem 3.1]{li2022homotopic}. To make $r_{\rho}<+\infty$, we need assume the initial distribution $\rho$ has full support on $\gS$, such as uniform distribution $\vu$ on $\gS$ that we have $r_{\vu} \leq |\gS|$. Now we could obtain sublinear/linear convergence of $Q_{s,a}(\pi^{(k)})$ from $V_{\rho}(\pi^{(k)})$ by Proposition \ref{prop:Q-V}.
\end{remark}

Additionally, if we have the convergence of policy, we could obtain the convergence of value function as well. 
Such derivation in Proposition \ref{prop:V-pi} also appears in the proof of \cite[Corollary 3.1]{li2022homotopic}.
\begin{prop}\label{prop:V-pi}
	We have the relationship between value function and policy as below:
	\[ V_{s}(\pi) - V_{s}^* \leq \frac{1}{(1-\gamma)^2} \norm{\pi(s,\cdot)-\pi^*(s, \cdot)}_1, \forall \pi^* \in \Pi^*. \]
	Hence,
	\[ V_{s}(\pi) - V_{s}^* \leq \frac{1}{(1-\gamma)^2} \norm{\pi-\Pi^*}_{\infty}.  \]
\end{prop}
\begin{proof}
	Using the performance difference lemma \cite{Kakade02approximatelyoptimal}, we obtain
	\begin{equation}
	\begin{aligned}
		V_{s}(\pi) - V_{s}^* &= \frac{1}{1-\gamma} \E_{s' \sim d_{s}(\pi)} \langle Q^*(s',\cdot), \pi(s,\cdot)-\pi^*(s, \cdot)\rangle \leq \frac{1}{1-\gamma} \norm{Q^*(s',\cdot)}_{\infty} \norm{\pi(s,\cdot)-\pi^*(s, \cdot)}_1 \\
		&\leq \frac{1}{(1-\gamma)^2} \norm{\pi(s,\cdot)-\pi^*(s, \cdot)}_1, \forall \pi^* \in \Pi^*.
	\end{aligned}
	\end{equation}
\end{proof}

\section{Faster Convergence of Policy}
\citet{xiao2022convergence} mentioned that under additional conditions, the PMD method may exhibit superlinear convergence. Recently, several works (\cite{li2022homotopic,cen2021fast,li2021quasi}) have shown the superlinear rates of PMD under (adaptive) regularization. Now we show that the superlinear rates still exist even for the unregularized problem without any additional assumptions (e.g., see the discussion in \cite[Section 4.3]{xiao2022convergence}). 

\begin{defn}[Inherited from \cite{khodadadian2021linear} Definitions 3.1 and 3.2]
	The set of of dummy states $\gS_d$ is defined as $\gS_d = \{s \in\gS \left. \right| Q^*(s, a) = V^*(s), \forall a \in \gA\}$. That is, $\gS_d$ is the set of states where playing any actions is optimal. 
	And we define the optimal advantage function gap $\Delta$ is defined as follows:
	\begin{equation}\label{eq:Delta}
		\Delta := \min_{s \not\in \gS_d}\min_{a \not\in\gA} \left(Q^*(s,a) - V^*(s)\right) =\min_{s \not\in \gS_d}\min_{a \not\in\gA_s^*} \left(Q^*(s,a) - \min_{a' \in\gA} Q^*(s,a')\right)>0,
	\end{equation}
	where $\gA_s^* = \{a \in\gA\left.\right| Q^*(s,a) = \min_{a' \in\gA} Q^*(s,a')\}$.
\end{defn}
 
\begin{remark}
	We need to underline that $\Delta \leq \frac{1}{1-\gamma}$, and $\Delta$ is MDP-dependent, which may be very small as our experiments shown.
\end{remark}

\begin{thm}\label{thm:frame}
	Consider the policy mirror descent with $\pi(0) \in \mathrm{ri}(\Pi)$.
	Suppose the step sizes are a non-decreasing sequence $\{\eta_k\} $ with $\eta_k>0, \forall k\geq 0$, and we already have 
	\begin{equation}\label{eq:linear}
		Q_{s,a}(\pi^{(k)}) - Q^*_{s, a} \leq A_k, \forall s \in \gS, a \in \gA,
	\end{equation}
	where $\{A_k\}$ is a decreasing sequence such that $\lim_{k \to +\infty} A_k=0$.
	Then we further have the following convergent results of policy.
	\begin{enumerate}
		\item Suppose $\nabla h(p)$ is well-defined for all $p \in \Delta(\gA)$, and we define $M := \sup_{a \in \gA, p \in \Delta(\gA)} |\nabla_a h(p)|$. Then if $K = \inf\{k \in \sN: \eta_k (\Delta-A_k)>4M \}<+\infty$, we have $\forall k \geq K,  s \in \gS$, $\pi_s^{(k+1)} \in \Pi_s^*$ and $V_s(\pi^{(k+1)}) = V_{s}^*$. That is, PMD with Bergman divergence define by $h(\cdot)$ could stop after \textbf{finite} iterations.
		\item Suppose $\partial h(p)$ is not well-defined for all $p \in \partial \Delta(\gA)$. Then we have $\forall s \in\gS_d, a \not\in \gA_s^*, b \in \gA_s^*$, for all $ k \geq 0, \pi_s^{(k+1)} \in \mathrm{ri}(\Delta(\gA))$ and
		\[ \nabla_a h(\pi_s^{(k+1)}) - \nabla_b h(\pi_s^{(k+1)})  \leq \nabla_a h(\pi_s^{(0)}) - \nabla_b h(\pi_s^{(0)}) - \left(\sum_{i=0}^k \eta_i \Delta - \sum_{i=0}^{k}\eta_i A_i\right). \]
	\end{enumerate}
\end{thm}

\begin{proof}
	The proofs mainly apply the arguments in Eq.~\eqref{eq:opt-cond1} and \eqref{eq:opt-cond}.
	\begin{enumerate}
		\item Since $ |\nabla_a h(p)| \leq M, \forall a \in \gA, p \in \Delta(\gS)$, then for any $s \in\gS_d, a \not\in \gA_s^*, b \in \gA_s^*$, we have
		\begin{equation}\label{eq:g-g}
			\left| \left(\nabla_a h(\pi_s^{(k+1)})-\nabla_b h(\pi_s^{(k+1)})\right) -\left(\nabla_a h(\pi_s^{(k)}) -\nabla_b h(\pi_s^{(k)})\right) \right| \leq 4M.
		\end{equation}
		Additionally, from Eq.~\eqref{eq:opt-cond}, we obtain
		\begin{equation}\label{eq:n-n}
			\eta_k \left[Q_{s,a}(\pi^{(k)})-Q_{s,b}(\pi^{(k)})\right] + \left[\nabla_a h(\pi_s^{(k+1)})-\nabla_b h(\pi_s^{(k+1)})\right] -\left[\nabla_a h(\pi_s^{(k)})-\nabla_b h(\pi_s^{(k)})\right] = n_{s,b}^{(k+1)}-n_{s,a}^{(k+1)}.
		\end{equation}
		for some $\vn = (n_{s,1}^{(k+1)},\dots,n_{s, |\gA|}^{(k+1)}) \in N_{\Delta(\gA)}(\pi_s^{(k+1)})$.
		Moreover, by the convergence of $Q_{s,a}(\pi^{(k)})$ in Eq.~\eqref{eq:linear}, we obtain
		\[ Q_{s,a}(\pi^{(k)})-Q_{s,b}(\pi^{(k)}) \geq Q_{s,a}^*-Q_{s,b}(\pi^{(k)}) \geq Q_{s,a}^*-Q_{s,b}^*-A_k \stackrel{\eqref{eq:Delta}}{\geq} \Delta-A_k, \]
		where the first inequality follows from $V_s(\pi) \geq V_s^*, \forall s \in \gS, \pi \in \Pi$.
		When $k \geq K$, we have
		\[ \eta_k \left[Q_{s,a}(\pi^{(k)})-Q_{s,b}(\pi^{(k)})\right] \geq \eta_k \left(\Delta-A_k\right) > 4M. \]
		Combining the above inequality with Eqs.~\eqref{eq:g-g} and \eqref{eq:n-n}, we get $n_{s,b}^{(k+1)}-n_{s,a}^{(k+1)}>0$. From the concrete expression of $ N_{\Delta(\gA)}(\pi_s^{(k+1)}) $ in Eq.~\eqref{eq:normal-cone}, we obtain $b \in \mathrm{supp}(\pi^{(k+1)}_s)$ and $a \not\in \mathrm{supp}(\pi^{(k+1)}_s)$.
		Hence, we conclude $\gA_s^* \subseteq \mathrm{supp}(\pi^{(k+1)}_s), (\gA_s^*)^c \subseteq \mathrm{supp}(\pi^{(k+1)}_s)^c$, that is $\mathrm{supp}(\pi_{s}^{(k+1)})=\gA_s^*$, i.e., $\pi_{s}^{(k+1)} \in \Pi_s^*$.

		\item Since $\partial h(p)$ is not well-defined on $p \in \partial \Delta(\gA)$.
		from \cite[Theorem 27.4]{rockafellar2015convex}, we can see $\forall k\geq 0,  \pi_s^{(k)} \in \mathrm{ri}(\Delta(\gA))$. 
		Now applying Eq.~\eqref{eq:opt-cond} and $\forall k\geq 0, \pi_{s}^{(k+1)} \in \mathrm{ri}(\Delta(\gA))$, we get
		\begin{equation}
			\eta_k Q_s(\pi^{(k)}) + \nabla h(\pi_s^{(k+1)})-\nabla h(\pi_s^{(k)}) = c_k \cdot \bm{1}_{|\gA|}
		\end{equation}
		for some $c_k \in\sR$. Telescoping from $k$ to $0$, we further derive that
		\[ \sum_{i=0}^k \eta_i Q_s(\pi^{(i)})+ \nabla_a h(\pi_s^{(k+1)}) - \nabla_a h(\pi_s^{(0)}) = \sum_{i=0}^k c_i \cdot \bm{1}_{|\gA|}. \]
		Similarly, for any $s \in\gS_d, a \not\in \gA_s^*, b \in \gA_s^*$, we have
		\begin{equation}\label{eq:g-a-b}
			\nabla_a h(\pi_s^{(k+1)}) - \nabla_b h(\pi_s^{(k+1)}) = \nabla_a h(\pi_s^{(0)}) - \nabla_b h(\pi_s^{(0)}) - \sum_{i=0}^k \eta_i \left[Q_{s,a}(\pi^{(i)})-Q_{s,b}(\pi^{(i)}) \right].
		\end{equation}
		Employing the convergence of $Q_{s,a}(\pi^{(k)})$ in Eq.~\eqref{eq:linear}, we have
		\begin{equation*}
		\begin{aligned}
			&\sum_{i=0}^k \eta_i \left[Q_{s,a}(\pi^{(i)})-Q_{s,b}(\pi^{(i)}) \right] \geq \sum_{i=0}^k \eta_i \left[Q_{s,a}^*-Q_{s,b}(\pi^{(i)}) \right]
			\stackrel{\eqref{eq:linear}}{\geq} \sum_{i=0}^k \eta_i \left[Q_{s,a}^*-Q_{s,b}^* - A_i\right] \stackrel{\eqref{eq:Delta}}{\geq} \sum_{i=0}^k \eta_i \Delta - \sum_{i=0}^{k}\eta_i A_i.
		\end{aligned}
		\end{equation*}
		Replacing the above bound to Eq.~\eqref{eq:g-a-b}, we finish the proof.
\end{enumerate}
\end{proof}

\begin{remark}
    We make some remarks for better understanding assumptions in Theorem \ref{thm:frame}.
    \begin{itemize}
    	\item Assumption 1 indicates that $\forall p \in \partial \Delta(\gA), \nabla h(p)$ exists, and $\nabla h(\cdot)$ is bounded everywhere on $\Delta(\gA)$. A sufficient condition for this assumption is that $\Delta(\gA) \subset \mathrm{int}(\mathrm{dom}(h))$ based on \cite[Theorem 23.4]{rockafellar2015convex}. The suitable divergences include Pearson $\chi^2$-divergence (the squared Euclidean distance), Tsallis divergence with entropic-index $q>1$.
    	\item Assumption 2 implies that $\partial h(p), \forall  p \in \partial \Delta(\gA)$ is ill-defined. 
    	That is, we obtain for some $a \in \gA$, $\lim_{p \to \partial \Delta(\gA)} \nabla_a h(p) \to \infty$. The suitable divergences include KL-divergence, Tsallis divergence with entropic-index $q<1$, squared Hellinger distance, $\alpha$-divergence, Jensen-Shannon divergence, Neyman $\chi ^{2}$-divergence.
    \end{itemize}
    However, we need to emphasis that not all these divergences are tractable in solving Eq.~\eqref{eq:pmd-pi}. 
    The common choices in previous work are the squared Euclidean distance and KL-divergence \cite{xiao2022convergence,li2022homotopic,lan2021policy}.
\end{remark}

In the first case of Theorem \ref{thm:frame}, we need for some $K\geq 0$, $\eta_K \Delta \geq \eta_K (\Delta-A_K)>4M $. Such a requirement could be satisfied for exponential steps sizes, but is unreasonable for constant learning rates ($\eta_k=\eta, \forall k \geq 0$) because we may have a large lower bound for $\eta = \eta_K \geq 4M/\Delta$ when $\Delta$ is small. We show that such constraint can be removed when applied to a smooth function $h(\cdot)$.

\begin{thm}\label{thm:better}
	Consider the policy mirror descent with $\pi(0) \in \Pi$.
	Suppose the step sizes are constant $\eta_k=\eta>0, \forall k\geq 0$, and Eq.~\eqref{eq:linear} holds.
	Assume that $h(\cdot)$ is continuously differentiable on $\Delta(\gA)$ and $L$-cocoercive under the norm $\norm{\cdot}$ on $\Delta(\gA)$:
	\[ \forall p,q \in \Delta(\gA), \langle \nabla h(p)-\nabla h(q), p-q \rangle \geq \frac{1}{L} \norm{\nabla h(p)-\nabla h(q)}^2. \]
	Then PMD with Bergman divergence define by $h(\cdot)$ could stop after \textbf{finite} iterations.
\end{thm}

\begin{proof}
	From Eq.~\eqref{eq:opt-cond1} with $p=\pi^{(k)}_s$, we can see
	\[ \langle \eta Q_s(\pi^{(k)}), \pi_s^{(k)}-\pi_s^{(k+1)} \rangle \geq \langle \nabla h(\pi_s^{(k+1)}) - \nabla h(\pi_s^{(k)}), \pi_s^{(k+1)}-\pi_s^{(k)} \rangle \geq \frac{1}{L} \norm{\nabla h(\pi_s^{(k+1)})-\nabla h(\pi_s^{(k)})}^2. \]
	Since $Q_{s,a}(\pi^{(k)}) \to Q_{s,a}^*$, then we get $\norm{\pi^{(k)}-\Pi^*}_{\infty} \to 0$, that is $\pi^{(k)}_{s,a} \to 0, \forall s \in\gS_d, a \not\in \gA_s^*$.
	Moreover, note that $Q_{s, b}^* = Q_{s, c}^*, \forall s \in\gS, b,c \in \gA_s^*$.
	Therefore, we obtain
	\begin{equation}\label{eq:q-pi}
		\begin{aligned}
			\lim_{k \to \infty} \langle \eta Q_s(\pi^{(k)}), \pi_s^{(k)}-\pi_s^{(k+1)}\rangle &= \lim_{k \to \infty} \sum_{a \not\in \gA_s^*} \eta Q_{s,a}(\pi^{(k)}) (\pi_{s,a}^{(k)}-\pi_{s,a}^{(k+1)}) + \sum_{b \in \gA_s^*} \eta Q_{s,b}(\pi^{(k)}) (\pi_{s,b}^{(k)}-\pi_{s,b}^{(k+1)}) \\
			&= 0 + \lim_{k \to \infty} \eta Q_{s,b}^* \sum_{b \in \gA_s^*}(\pi_{s,b}^{(k)}-\pi_{s,b}^{(k+1)}) = \eta Q_{s,b}^*(1-1) = 0.
		\end{aligned}
	\end{equation}
	Hence, we derive that
	\begin{equation}\label{eq:grad-k}
		\lim_{k \to \infty} \nabla h(\pi_s^{(k+1)})-\nabla h(\pi_s^{(k)}) = 0. 
	\end{equation}
	By the convergence of $Q_{s,a}(\pi^{(k)})$ in Eq.~\eqref{eq:linear}, we obtain
	\begin{equation}\label{eq:cond-2}
		Q_{s,a}(\pi^{(k)})-Q_{s,b}(\pi^{(k)}) \geq Q_{s,a}^*-Q_{s,b}(\pi^{(k)}) \geq Q_{s,a}^*-Q_{s,b}^*-A_k \stackrel{\eqref{eq:Delta}}{\geq} \Delta-A_k=\Delta+o(1).
	\end{equation}
	Then from Eq.~\eqref{eq:opt-cond}, we obtain $\forall s \in\gS_d, a \not\in \gA_s^*, b \in \gA_s^*$,
	\begin{equation}\label{eq:n-n-1}
		\begin{aligned}
			n_{s,b}^{(k+1)}-n_{s,a}^{(k+1)} & \ = \eta_k \left[Q_{s,a}(\pi^{(k)})-Q_{s,b}(\pi^{(k)})\right] + \left(\nabla_a h(\pi_s^{(k+1)})-\nabla_a h(\pi_s^{(k)})\right) -\left(\nabla_b h(\pi_s^{(k+1)}) -\nabla_b h(\pi_s^{(k)})\right) \\
			&\stackrel{\eqref{eq:grad-k}}{=} \eta_k \left[Q_{s,a}(\pi^{(k)})-Q_{s,b}(\pi^{(k)})\right]+o(1)
			\stackrel{\eqref{eq:cond-2}}{\geq} \eta\Delta+o(1).
		\end{aligned}
	\end{equation}
	Since $|\gS|, |\gA|$ are finite, we get for a large enough $K$, we have that $ n_{s,b}^{(K+1)} - n_{s,a}^{(K+1)} > 0, \forall s \in\gS_d, a \not\in \gA_s^*, b \in \gA_s^* $. The remaining proof is the same the first case of Theorem \ref{thm:frame}.
\end{proof}

\begin{remark}
	We need to emphasize that although we extend the suitable scope of learning rates in Theorem \ref{thm:better}, but we do not give the explicit steps for finite-step convergence as Theorem \ref{thm:frame} shown.
	Additionally, the satisfied function $h(\cdot)$ in Theorem \ref{thm:better} includes all twice continuously differential convex functions on $\Delta(\gA)$. We show the reason below.  Note that
	\[ \forall p,q \in \Delta(\gA), \nabla h(p)-\nabla h(q) = \left[\int_{0}^1 \nabla^2 h(q+t(p-q))dt\right] (p-q)  := J(p,q)(p-q). \]
	Since $\Delta(\gA)$ is a compact set and $h(\cdot)$ is a twice continuously differential convex function, we have for some $M > 0$ that $M \cdot I \succeq \nabla^2 h(p) \succeq 0, \forall p \in \Delta(\gA)$. Thus we get $M \cdot I \succeq J(p-q) \succeq 0$. Therefore
	\[ \langle \nabla h(p)-\nabla h(q), p-q\rangle =(p-q)^\top J(p,q) (p-q) \geq \frac{1}{M} (p-q)^\top J(p,q)^\top J(p,q)(p-q) = \frac{1}{M}\norm{\nabla h(p)-\nabla h(q)}_2^2. \]
	Hence, $h(x)$ are $M$-cocoercive under the norm $\norm{\cdot}_2$ on $\Delta(\gA)$.
\end{remark}

Finally, from Proposition \ref{prop:V-pi}, we also can convert the convergence of policy in Theorem \ref{thm:frame} to value function. We omit the detail here, but show the numerical results to support our viewpoint.

\subsection{An Example of Euclidean Distance}
The new discovery in our note appears in the finite-step-convergence of policy.
To give a better understanding, we show the example of Euclidean distance here.
We have the update rule as 
\[ \pi_s^{(k+1)} = \bm{\mathrm{proj}}_{\Delta(\gA)} \left\{\pi_s^{(k)} - \eta_k Q_s(\pi^{(k)})\right\}, \forall s \in \gS. \]
The projection to probability simplex has the formulation as
\[ \vx^+ = \bm{\mathrm{proj}}_{\Delta(\gA)} \{\vx\} = (\vx+\alpha\bm{1})_+ \]
for some $\alpha \in \sR$ to satisfy $\langle \vx^+, \bm{1} \rangle =1$, and $(x)_+=\max\{0, x\}$.
Hence, we can see for any $i \neq j \in [|\gA|]$, if $x_i-x_j>1$, then $x_j^+=0$. Otherwise, $(x_j+\alpha)_+ = x_j^+>0$, we get $x_i+\alpha > x_j+\alpha > 0$, and $x_i^+ = x_i+\alpha > 1 + x_j + \alpha > 1$, which contradicts the requirement $\langle \vx^+, \bm{1} \rangle =1$ since $\forall k \in [|\gA|], x_k^+ \geq 0$.

Now we turn back to the update of $\pi_s^{(k)}$. Since from previous work \cite{xiao2022convergence}, we already have the convergence of $Q_s(\pi^{(k)}) \to Q_s^*$. Hence, for large enough $k$, any $a \not\in \gA_s^*, b \in \gA_s^*$ we have that $Q_{s,a}(\pi^{(k)}) -  Q_{s,b}(\pi^{(k)}) = Q_{s,a}^* - Q_{s,b}^* +o(1) \geq \Delta + o(1)$. Therefore, we derive that
\[ \eta_k \left(Q_{s,a}(\pi^{(k)}) -  Q_{s,b}(\pi^{(k)})\right) = \eta_k \left(Q_{s,a}^* - Q_{s,b}^*+o(1)\right) \geq \eta_k \left(\Delta+o(1)\right). \]
Thus, for large enough $k$ and $\eta_k$ to guarantee $\eta_k \left(\Delta+o(1)\right)>1$, we have $\eta_k \left(Q_{s,a}(\pi^{(k)}) -  Q_{s,b}(\pi^{(k)})\right)>1$ leading to $\pi^{(k+1)}_{s, a}=0$, i.e., $\mathrm{supp}(\pi^{(k+1)}_s) \in \gA_s^*$. Hence, we conclude that $\pi^{(k+1)} \in \Pi^*$.

\subsection{Application to Common Divergence}
Next, we see some common Bergman divergences and give explicit rates of policy.
Though we mention generalized Bregman divergence defined by a general convex function $h(\cdot)$, we usually adopt continuously differentiable $h(\cdot)$ in practice. 

\begin{enumerate}
	\item The squared Euclidean distance: $h(p) = \frac{1}{2}\norm{p}_2^2, D(p,p')=\frac{1}{2}\norm{p-p'}^2$, which satisfies the first case of Theorem \ref{thm:frame} with $M=1$. That is, policy converges to an optimal policy after finite steps. 
	The update rule is
	\[ \pi_s^{(k+1)} = \bm{\mathrm{proj}}_{\Delta(\gA)} \left\{\pi_s^{(k)} - \eta_k Q_s(\pi^{(k)})\right\}, \forall s \in \gS. \]
	\begin{itemize}
		\item Constant learning rate $\eta_k = \eta \geq \frac{8}{\Delta}$. Note that $D_0^* \leq 1, A_k = \frac{\gamma r_{\rho}}{k+1}\left(\frac{1}{\eta(1-\gamma)}+\frac{1}{(1-\gamma)^2}\right)$. To guarantee $\eta_k(\Delta-A_k) > 4M$, we only need $A_k \leq \Delta/2$ and $\eta_k\Delta>8M$. Hence, we obtain $K = \left\lceil \frac{2r_{\rho}}{\Delta}\left(\frac{1}{\eta(1-\gamma)}+\frac{1}{(1-\gamma)^2}\right) \right\rceil = O\left(\frac{r_{\rho}}{\Delta(1-\gamma)^2}\right)$.
		\item Exponential learning rate $\eta_{k} = \left(\frac{\vartheta_\rho}{\vartheta_\rho-1}\right)^k \cdot \eta_0, \eta_0=O(1)$. We have $A_k = \gamma r_{\rho}\left(1-\frac{1}{\vartheta_\rho}\right)^k \left(\frac{1}{1-\gamma}+\frac{1}{\eta_0\gamma}\right)$.
		Solving $\eta_k(\Delta-A_k) > 4M $, we obtain
		$K= \left\lceil\vartheta_{\rho}\ln \frac{4+\gamma r_{\rho}\left(\frac{1}{1-\gamma}+\frac{1}{\eta_0\gamma}\right)}{\eta_0\Delta} \right\rceil = O\left(\vartheta_{\rho}\ln \frac{ r_{\rho} }{\eta_0\Delta(1-\gamma)} \right)$.
	\end{itemize}
	\item Kullback-Leibler (KL) divergence $h(p) = \sum_{a \in \gA}p_a\log p_a, D(p,p')=\sum_{a \in \gA}p_a\log \frac{p_a}{p_a'}$, which satisfies the second case of Theorem \ref{thm:frame} because $\lim_{p_a \to 0} \partial_a h(p) = \lim_{p_a \to 0} \log p_a+1 = -\infty$. The update rule is
	\[ \pi_s^{(k+1)} \propto \pi_s^{(k)} \cdot e^{- \eta_k Q_s(\pi^{(k)})}. \]
	For brevity, we adopt $\pi_s^{(0)}$ as the uniform distribution on $\gA, \forall s \in\gS$ in the following.
	\begin{itemize}
		\item Constant learning rate $\eta_k = \eta $. Note that $D_0^* \leq \ln |\gA|, A_k = \frac{\gamma r_{\rho}}{k+1}\left(\frac{\ln |\gA|}{\eta(1-\gamma)}+\frac{1}{(1-\gamma)^2}\right)$. Then we obtain
		\[ \ln \frac{\pi_{s, a}^{(k)}}{\pi_{s, b}^{(k)}} \leq -\left(\sum_{i=0}^{k-1} \eta_i \Delta - \sum_{i=0}^{k-1}\eta_i A_i\right) = -k\eta\Delta+(1+\ln k) \gamma r_{\rho}\left(\frac{\ln |\gA|}{\eta(1-\gamma)}+\frac{1}{(1-\gamma)^2}\right). \]
		Hence, we obtain
		\[ \norm{\pi^{(k)}-\Pi^*}_{\infty} \leq 2|\gA| e^{-k\eta\Delta(1+o(1))}. \]
	    A similar result also appears in \citet{khodadadian2021linear}.
		\item Exponential learning rate $\eta_{k} = \left(\frac{\vartheta_\rho}{\vartheta_\rho-1}\right)^k \cdot \eta_0, \eta_0=O(1)$. Note that $A_k = \gamma r_{\rho}\left(1-\frac{1}{\vartheta_\rho}\right)^k \left(\frac{\ln |\gA|}{1-\gamma}+\frac{1}{\eta_0\gamma}\right)$.
		\[ \ln \frac{\pi_{s, a}^{(k)}}{\pi_{s, b}^{(k)}}  \leq -\left(\sum_{i=0}^{k-1} \eta_i \Delta - \sum_{i=0}^{k-1}\eta_i A_i\right) \leq k \eta_0 \gamma r_{\rho} \left(\frac{\ln |\gA|}{1-\gamma}+\frac{1}{\eta_0\gamma}\right) - \left(\frac{\vartheta_\rho}{\vartheta_\rho-1}\right)^k \cdot \eta_0 \Delta. \]
		Hence, we obtain
		\[ \norm{\pi^{(k)}-\Pi^*}_{\infty} \leq 2|\gA| e^{-\left(\frac{\vartheta_\rho}{\vartheta_\rho-1}\right)^k\eta_0\Delta(1+o(1))}. \]
		A similar result also appears in \citet{li2022homotopic}.
	\end{itemize}
	\item Tsallis divergence: $h(p) = \frac{\sum_{a \in \gA} p_a^q-1}{q-1}, q>0$. When $q \to 1$, the Tsallis entropy converges to the negative Shannon entropy. 
	Generally, the update rule under this divergence is \textcolor{red}{intractable}:
	\[ \pi_s^{(k+1)} = \left[(\pi_s^{(k)})^{q-1}-\frac{\eta_k Q_s(\pi^{(k)})}{q}+\vn_s^{(k)}\right]^{\frac{1}{q-1}}, \pi_s^{(k+1)} \in \Delta(\gA). \]
	We briefly mention the results following Theorem \ref{thm:frame}.
	\begin{itemize}
		\item $q>1$, finite step stops.
		\item $0<q<1$, 
		\begin{itemize}
		    \item Constant learning rates: $\norm{\pi^{(k)}-\Pi^*}_{\infty} = O\left(\left(\frac{1}{k}\right)^{\frac{1}{1-q}}\right)$; 
		    \item Exponential learning rates: $\norm{\pi^{(k)}-\Pi^*}_{\infty} = O\left(\left(1-\frac{1}{\vartheta_\rho}\right)^{\frac{k}{1-q}}\right)$
		\end{itemize}
	\end{itemize}
	\item Other common divergences also seem intractable.
\end{enumerate}

\begin{table}
	\centering
	\begin{tabular}{cccc}
		\toprule
		Bregman divergence & Learning rate & Rates & Complexity in Each Step \\
		\midrule
		KL & Constant    & sublinear + linear  & $O(|\gS|\cdot|\gA|)$ \\
		KL & Exponential & linear + superlinear & $O(|\gS|\cdot|\gA|)$ \\
		\midrule
		Euclidean & Constant    & sublinear + finite steps & $O(|\gS|\cdot|\gA|\ln|\gA|)$ \\
		Euclidean & Exponential & linear + finite steps & $O(|\gS|\cdot|\gA|\ln|\gA|)$ \\
		\midrule
		Tsallis & Constant    & sublinear (+ finite steps if $q>1$)  & intractable \\
		Tsallis & Exponential & linear (+ finite steps if $q>1$) & intractable \\
		\bottomrule
	\end{tabular}
	\caption{Summary of convergence under different steps sizes and Bregman divergence.}
\end{table}

\section{Experiments}
In this section, we conduct toy experiments to verify our findings of convergence.
We construct a MDP by randomly sampling $P, R$ with entries from $\mathrm{Unif}(0,1)$, and then normalize $P$ to be a probability transition matrix.
We set $|\gS|=20, |\gA|=100$ and choose the target $V_{\rho}(\cdot)$ with $\rho=\bm{1}_{|\gS|}, \gamma=0.999$ through our experiments.
We adopt two kinds of learning rates. One is exponential learning rates with $\eta_k = \gamma^{-k}\eta_0, \forall k\geq 0$, another is constant learning rates $\eta_k=\eta_0, \forall k\geq0$.
For simplification, we adopt uniform distribution as the initial policy, i.e., $\pi_{s,a}^{(0)}=1/|\gA|, \forall s \in \gS, a \in \gA$.
We also list some related quantities in our theorem: $\Delta = 0.0012, \ln\frac{1}{\gamma} = 0.0010005$.
Since the MDP with adaptive regularization has similar rates \cite{lan2021policy, li2022homotopic}, we also try these methods with same settings for fair comparison.
We choose the regularization in each step as $\{\tau_k\}$, and $1+\eta_k\tau_k = 1/\gamma, \forall k \geq 0$ following \cite[Theorem 2.1]{li2022homotopic}. The update of regularized MDP is 
\begin{equation}\label{eq:pmd-pi-reg}
    \pi_s^{(k+1)} = \mathop{\arg\min}_{p \in \Delta(\gA)} \left\{ \eta_k \left[\langle Q_s(\pi^{(k)}), p\rangle + \tau_k D(p,\pi^{(0)}_s) \right] + D(p, \pi_s^{(k)}) \right\}, \forall s \in \gS.
\end{equation}
Then for Euclidean distance, the update rule is
\[ \pi_s^{(k+1)} = \bm{\mathrm{proj}}_{\Delta(\gA)} \left\{\frac{\pi_s^{(k)} - \eta_k Q_s(\pi^{(k)})+\eta_k\tau_k\pi^{(0)}_s}{1+\eta_k\tau_k} \right\}, \forall s \in \gS. \]
And for KL divergence, the update rule is
\[ \pi_s^{(k+1)} \propto (\pi_s^{(k)} \cdot e^{-\eta_k Q_s(\pi^{(k)})})^{1/(1+\eta_k\tau_k)} , \forall s \in \gS. \]

\begin{figure}[t]
	\centering
	\begin{subfigure}[b]{0.48\textwidth}
		\includegraphics[width=\linewidth]{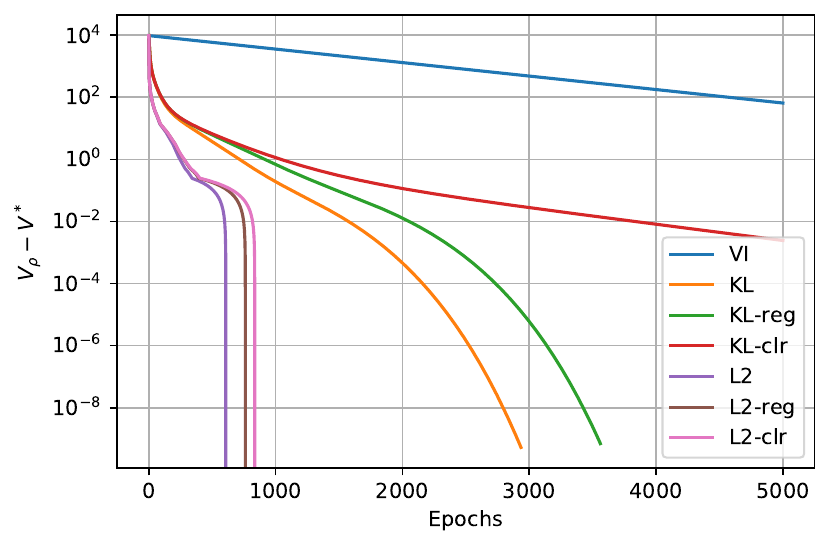}
		\caption{$\eta_0=1$.} \label{fig:eta-1}
	\end{subfigure}
	\begin{subfigure}[b]{0.48\textwidth}
		\includegraphics[width=\linewidth]{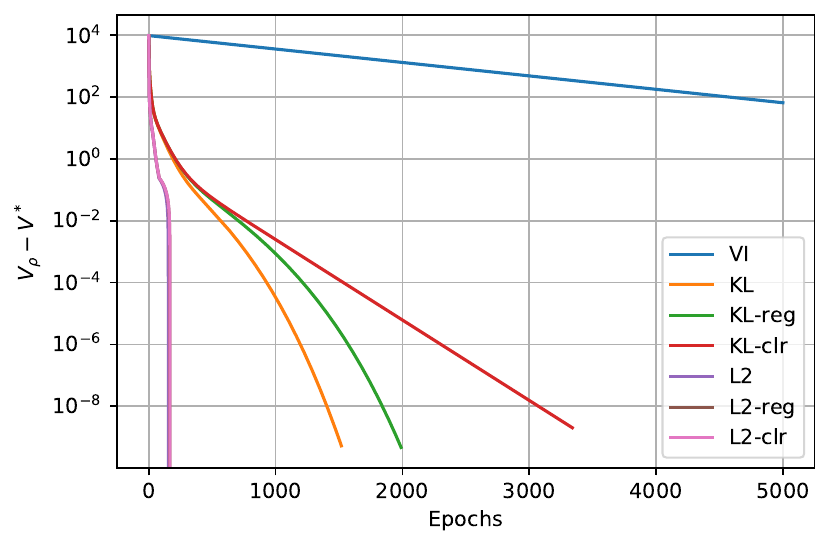}
		\caption{$\eta_0=5$.} \label{fig:eta-5}
	\end{subfigure}
	\label{fig:rate}
	\caption{L2: the squared Euclidean distance; KL: Kullback-Leibler (KL) divergence. Default: exponential learning rates. `reg': using adaptive regularization $\tau_k$ with $1 + \eta_k\tau_k = 1/\gamma$. VI: value iteration. `clr': constant learning rate.}
\end{figure}

We briefly show some observation based on our experiments.
\begin{enumerate}
	\item KL divergence 
	\begin{itemize}
		\item From $\eta_0\Delta \geq \ln\frac{1}{\gamma}$, we have `KL-clr' is faster than VI for large $k$.
		\item Under the same (exponential) learning rates, employing regularization has less benefit than no regularization.
		\item Convergence rates under constant learning rates: sublinear + linear, under exponential learning rates: linear + superlinear.
	\end{itemize}
	\item L2 distance ($\chi^2$-divergence)
	\begin{itemize}
		\item Finite step convergence.
		\item Under the same (exponential) learning rates, employing regularization has less benefit than no regularization.
		\item Using L2 is clearly faster than KL divergence, but each step of update complexity is not the same ($\log|\gA|$ worse).
	\end{itemize}
\end{enumerate}

\bibliography{reference}
\bibliographystyle{plainnat}
\end{document}